\documentclass[a4paper]{amsart}

\usepackage[english]{babel}

\usepackage{thmtools}
\usepackage{amsmath}
\usepackage{amssymb}
\usepackage{amsthm}
\usepackage{amsfonts}
\usepackage{graphicx}
\usepackage{comment}
\usepackage{enumitem}
\usepackage{color}
\usepackage[colorlinks=true, allcolors=blue]{hyperref}
\usepackage{cleveref}
\usepackage{todonotes}
\usepackage{oldgerm}
\usepackage{longtable}
\usepackage{tikz}

\declaretheorem[name=Theorem,numberwithin=section]{thm}

\newtheorem{Lemma}[thm]{Lemma}

\newtheorem{claim}{Claim}[thm]
\newtheorem*{claim*}{Claim}

\theoremstyle{definition}
\newtheorem{defn}[thm]{Definition}
\newtheorem{Def}[thm]{Definition}

\newtheorem{fact_n_def}[thm]{Fact and Definition}
\newtheorem{Fact}[thm]{Fact}

\newcommand{\subd}{\mathrel{\triangleleft}}

\newcommand{\Lk}{\operatorname{Lk}}

\renewcommand{\mid}{\,:\,}
\newenvironment{enumerate-(a)}{\begin{enumerate}[label={\upshape (\alph*)}, leftmargin=2pc]}{\end{enumerate}}
\newenvironment{enumerate-(a)-r}{\begin{enumerate}[label={\upshape (\alph*)}, leftmargin=2pc,resume]}{\end{enumerate}}
\newenvironment{enumerate-(a)-5}{\begin{enumerate}[label={\upshape (\alph*)}, leftmargin=2pc,start=5]}{\end{enumerate}}
\newenvironment{enumerate-(A)}{\begin{enumerate}[label={\upshape (\Alph*)}, leftmargin=2pc]}{\end{enumerate}}
\newenvironment{enumerate-(A)-r}{\begin{enumerate}[label={\upshape (\Alph*)}, leftmargin=2pc,resume]}{\end{enumerate}}
\newenvironment{enumerate-(i)}{\begin{enumerate}[label={\upshape (\roman*)}, leftmargin=2pc]}{\end{enumerate}}
\newenvironment{enumerate-(i)-r}{\begin{enumerate}[label={\upshape (\roman*)}, leftmargin=2pc,resume]}{\end{enumerate}}
\newenvironment{enumerate-(I)}{\begin{enumerate}[label={\upshape (\Roman*)}, leftmargin=2pc]}{\end{enumerate}}
\newenvironment{enumerate-(I)-r}{\begin{enumerate}[label={\upshape (\Roman*)}, leftmargin=2pc,resume]}{\end{enumerate}}
\newenvironment{enumerate-(1)}{\begin{enumerate}[label={\upshape (\arabic*)}, leftmargin=2pc]}{\end{enumerate}}
\newenvironment{enumerate-(1)-r}{\begin{enumerate}[label={\upshape (\arabic*)}, leftmargin=2pc,resume]}{\end{enumerate}}

\newenvironment{enumerate-(star)}{\begin{enumerate}[label={\upshape{(\( \star_{ \arabic*} \))}}, leftmargin=2pc]}{\end{enumerate}}

\renewcommand{\o}{\omega}
\newcommand{\lo}{<\!\omega{}}

\newcommand{\vv}{\mathbf{v}}

\newcommand{\ee}{\mathbf{e}}

\newcommand{\R}{\mathbb{R}} 
 
\newcommand{\N}{\mathbb{N}} 

\newcommand{\UU}{\mathcal{U}}

\newcommand{\bDelta}{\mathbf{\Delta}}

\newcommand{\interior}{\operatorname{int}}
\newcommand{\scl}{\operatorname{scl}}
\renewcommand{\int}{\operatorname{int}}

\allowdisplaybreaks

\mathchardef\mhyphen="2D
\newcommand{\rest}{\!\restriction\!}

\newcommand{\St}{\operatorname{St}}

\newcommand{\es}{\varnothing}
\renewcommand{\ge}{\geqslant}
\renewcommand{\le}{\leqslant}

\renewcommand{\geq}{\geqslant}

\renewcommand{\Cap}{\bigcap}
\renewcommand{\Cup}{\bigcup}

\newcommand{\cat}{{}^\frown}

\renewcommand{\subseteq}{\subset}

\newcommand{\obDelta}{\overset{\circ}{\mathbf{\Delta}}{}}

\title{Alexander's conjecture for infinite simplicial complexes} 


\author[M.~Iannella]{Martina Iannella} 
\address{Department of Discrete Mathematics and Geometry, TU Wien.~Wiedner Hauptstr.~8--10, 1040 Vienna, Austria}
\email{martina.iannella@tuwien.ac.at}
\author[V.~Weinstein]{Vadim Weinstein} 
\address{Center for Ubiquitous Computing\\
Erkki Koiso-Kanttilan katu 3\\ 
door E P.O Box 4500\\
FI-90014 University of Oulu} 
\email{vadim.weinstein@iki.fi}

\thanks{The first author was supported by the Austrian Science Fund (FWF) [10.55776/ESP1625325]. The second author was supported by a European Research Council Advanced Grant (ERC AdG, ILLUSIVE: Foundations of Perception Engineering, 101020977)}

\begin{document}

\begin{abstract}
  Alexander's conjecture states that for every two finite triangulations of the same topological space, if they have a common subdivision, then they have a common stellar subdivision. We generalize the recent result of Adiprasito and Pak, who resolved Alexander's conjecture for finite simplicial complexes, to infinite simplicial complexes.
\end{abstract}

\maketitle

\tableofcontents

\section{Introduction}

The classical problem of comparing triangulations of a fixed topological space via local combinatorial moves goes back to the works of Alexander and Newman, where stellar subdivisions as elementary operations on simplicial complexes were introduced (see \cite{Ale30, New26, New31}). These ideas became foundational in piecewise-linear (PL) topology, and were further developed and formalized in standard references such as \cite{hudson1969pl,Lickorish1999}. In this setting, stellar subdivisions generate a natural notion of refinement between triangulations.

A central theme in this theory is the relationship between common subdivisions and sequences of stellar moves. In particular, a conjectural principle, often referred to as \textit{Alexander’s conjecture} in combinatorial topology, asserts that if two triangulations of the same space admit a common subdivision, then they should also admit a common stellar subdivision.

In the finite setting, this conjecture has recently been resolved by Adiprasito and Pak in \cite{AP24}.

The goal of this paper is to extend this result to infinite locally finite simplicial complexes. Working in a fixed ambient space $X$, we consider triangulations which may be infinite but locally finite in the usual sense. 
We appropriately define stellar subdivisions of such
triangulations which are obtained as the result of
an infinite sequence of stellar subdivisions by
a locally finite sequence. We call them $\omega$-stellar subdivisions.
The main result shows that the existence of a common subdivision implies the existence of a common 
$\omega$-stellar subdivision.
See Definition \ref{def:StellarSubdivision} for precise treatment. The local finiteness assumption ensures that the resulting complexes are well-defined and agree with the usual finite stellar subdivision on compact subsets.

\begin{thm}\label{thm:intro_main}
  Let $X$ be a topological space and let $S$ and $T$ be 
  (possibly infinite) triangulations of $X$ which admit a 
  common subdivision. Then there exist locally finite 
  sequences $\bar s$ and $\bar t$ in \(X\) such that
  \[
    S * \bar s \;=\; T * \bar t,
  \]
  where \(S*\bar s\) and \(T*\bar t\) denote the \(\omega\)-stellar subdivisions of \(S\)  and \(T\) determined by \(\bar s\) and \(\bar t\), respectively.
\end{thm}

The paper is organized as follows. Section~\ref{sec:Preliminaries} contains the necessary background on simplicial complexes and stellar subdivisions, including the definition and basic properties of locally finite subdivision sequences. Section~\ref{sec:Main result} contains the proof of the main theorem.

\section{Notation, terminology and preliminary results}\label{sec:Preliminaries}

By $A\subset B$ we mean that $A$ is a subset of $B$, possibly $A=B$.
By $|f|$ we denote
the range of the map $f\colon X \to Y$. Interchangeably, we
denote the composition of functions $f\colon X\to Y$ and
$g\colon Y\to Z$ either by $g\circ f$ or~$gf$.
We denote both finite
and infinite sequences as $\bar z=(z_i)_{i\in I}$ for e.g. $I=\N$. 
By $l(\bar z)$ we denote the length of $\bar z$,
and we start indexing from
$0$ unless mentioned otherwise.
If
each element of the sequence is an element of $A$, we denote
$\bar z\subset A$. Slightly abusing notation, $|\bar z|$ is
the range $\{z_i\mid i\in I\}$ of the sequence.
The concatenation of sequences is denoted by $\bar z\cat \bar w$.
If $\bar s=\bar z\cat \bar w$, we write $\bar z\subset \bar s$
and say that $\bar s$ is an \textbf{end-extension} of $\bar z$.
If $\bar s^0\subset \bar s^1\subset \cdots$ is a sequence of
sequences such that $\bar s^{i+1}$ is an end-extension of $\bar s^i$,
then we denote by $\Cup_{i\in\N}\bar s^i$ the sequence $\bar s$
whose $j$-th element equals the $j$-th element of some (any) $\bar s^i$
with $l(\bar s^i)>j$.

If $X$ is a topological space and $Y \subseteq X$, then by $\interior_X(Y)$ and $\partial_X Y$ we denote, respectively, the interior and boundary of $Y$ in $X$. If $X$ is clear from the context, it is dropped from subscript. Sometimes the interior of $Y$ is denoted also by $\overset{\circ}{Y}$. The closure of $Y$ is denoted by~$\overline{Y}$.

\begin{defn}\label{def:Simplex}
  The \textbf{standard $n$-simplex} is the convex hull of the standard
  basis $\{\ee_k\mid 0\le k\le n\}$ of $\R^{n+1}$ denoted
  $\bDelta^n$. An \textbf{$n$-simplex} in a topological space $X$ is a
  (continuous) embedding $\kappa\colon \bDelta^n\to X$. The number $n$
  is the \textbf{dimension} of $\kappa$ and is denoted by
  $\dim(\kappa)$.  If $X=\R^m$ and $\kappa$ is linear, then $\kappa$
  is a \textbf{rectilinear} $n$-simplex.
\end{defn}

\begin{defn}\label{def:Faces}
  A \textbf{standard (proper)} $k$-\textbf{face} of the standard
  $n$-simplex is the convex hull of a (proper) subset of
  $\{\ee_0,\dots,\ee_n\}$ of size $k+1$. Given a simplex
  $\kappa\colon\bDelta^n\to X$ in a topological space $X$, its
  \textbf{$k$-face} is a $k$-simplex $\lambda\colon \bDelta^k\to X$
  such that $\kappa^{-1}\lambda$ is an
  isometry from $\bDelta^k$ onto a standard $k$-face of $\bDelta^n$
  (in particular $\kappa^{-1}\lambda$ is rectilinear). The $0$-faces
  (both standard and otherwise) are called
  \textbf{vertices}. Technically a vertex is a function from the
  singleton $\{\ee_k\}$, for $k \in \{0, \dots, n\}$, to $X$, but we
  usually identify it with the unique point $\kappa(\ee_k)$ in its
  range.  The set of vertices of a simplex $\kappa$ is denoted by
  $V(\kappa)$.  Let $\partial\bDelta^n$ be the union of all standard
  proper faces of $\bDelta^n$. Let
  $\obDelta^n=\bDelta^n\setminus \partial\bDelta^n$.  For any simplex
  $\kappa\colon \bDelta^n\to X$, denote
  \begin{equation}
    \overset{\circ}{\kappa}:=\kappa\rest \obDelta^n.
    \label{eq:oversetcirc}
  \end{equation}
  Note that according to this definition, for a vertex $v$, we have
  $\overset{\circ}v=v$
  , because the only proper face of 
  a vertex is empty.
\end{defn}

\begin{defn}\label{def:EquivalenceOfsimplices}
  Two $n$-simplices $\kappa,\lambda\colon \bDelta^n\to X$ are considered
  \textbf{equivalent}, if $|\kappa|=|\lambda|$ and
  $\kappa^{-1}\lambda$ (and hence also $\lambda^{-1}\kappa$)
  is an isometry of $\bDelta^n$ onto itself.  \textbf{From now on we
    consider simplices upto this equivalence relation unless stated
    otherwise.}
\end{defn}

\begin{fact_n_def}\label{def:DeltaInRm}
  A rectilinear simplex $\kappa\colon \bDelta^n\to \R^m$ is uniquely
  (upto the equivalence relation defined above) determined by its
  vertices $V(\kappa)=\{\kappa(\ee_i)\mid 0\le i\le n\}$.  Conversely,
  if $W=\{\vv_0,\dots,\vv_n\}$ are points in $\R^m$, there is a unique
  linear map from $\bDelta^n$ to $\R^m$ with $\ee_i\mapsto \vv_i$ for
  all $0\le i \le n$.  This map is denoted by
  $\Delta[W]=\Delta[\vv_0,\dots,\vv_n]$.  Note that $|\Delta[W]|$ is
  the convex hull of~$W$. The map $\Delta[W]$ is a rectilinear
  $n$-simplex if and only if the vectors $\vv_i-\vv_0$ are linearly
  independent for $1\le i\le n$. In this case, clearly
  $V(\Delta[W])=W$. See also Definition~\ref{def:DeltaInComplex} for a
  generalization of the operation $\Delta$ for any simplex.
\end{fact_n_def}

\begin{defn}\label{def:simplicial_complex}\cite[Chapter III]{hudson1969pl}
  Let \(X\) be a topological space. A \textbf{simplicial complex in
    \(X\)} is a (possibly finite) countable set $T$ of simplices in
  $X$ such that the following hold up to the equivalence relation
  of Definition~\ref{def:EquivalenceOfsimplices}:
  \begin{enumerate-(i)}
  \item\label{def:simplicial_complex-1} Every face of every simplex in
    $T$ is in $T$.
  \item\label{def:simplicial_complex-2} If $\kappa_0,\kappa_1\in T$
    and $|\kappa_0|\cap |\kappa_1|\ne\es$, then there is $\lambda\in T$
    such that $|\lambda|=|\kappa_0|\cap |\kappa_1|$ and $\lambda$
    is a face of both $\kappa_0$ and $\kappa_1$.
  \item\label{def:simplicial_complex-3} For all
    $\kappa\in T$ there is a neighbourhood $U\subset X$ of $|\kappa|$
    such that there are at most finitely many $\lambda\in T$
    with $|\lambda|\cap U\ne\es$.
  \end{enumerate-(i)}
  Sometimes we call a simplicial complex just \textbf{complex}. If
  $T_0\subseteq T$ is also a simplicial complex, then $T_0$ is called
  a \textbf{subcomplex} of~$T$.  We denote by
  \(V(T)=\Cup_{\kappa\in T} V(\kappa)\) the set of all $0$-simplices
  (vertices) of~\(T\). The set $|T| = \Cup_{\kappa\in T} |\kappa|$ is the \textbf{realization}
  of~$T$ and is endowed with the topology induced from \(X\). If
  $|T|=X$, we say that $T$ is a \textbf{triangulation} of~$X$.  Thus,
  $T$ is always a triangulation of $|T|$. 
  Note that condition \ref{def:simplicial_complex-3} implies that
  $|T|$ is locally compact.
  
  If $S\subset T$,
  let $\scl(S)$ be the smallest subcomplex of $T$ which
  contains~$S$. So, for a simplex $\kappa$, by $\scl(\kappa)$ we denote the
  complex which consists of $\kappa$ and all its faces. A special case
  is $\scl(\bDelta^n)$ which consists of
  the identity map $\bDelta^n\to\bDelta^n$
  and \textit{its} faces.
\end{defn}

\begin{defn}\label{def:rel_complexes}
  Let \(S\) and $T$ be simplicial complexes. We say that $S$ is a \textbf{subdivision} of $T$, denoted
    $S\subd T$, if $|S|=|T|$, and for all $\sigma\in S$
    there is $\kappa\in T$ such that $|\sigma|\subseteq |\kappa|$ and
    $\kappa^{-1}\sigma$ is rectilinear.
\end{defn}

The following is a collection of useful observations.

\begin{Lemma}\label{lemma:Collection1}
  \begin{enumerate-(i)}
  \item  \label{fact:SubsetImpliesSimplex}
    Suppose $S$ and $T$ are complexes and $S\subset T$.  If
    $\tau\in T$ is such that $|\tau|\subset |S|$, then $\tau\in S$.
  \item \label{fact:SubcomplexEquals}
    Suppose $S$ and $T$ are complexes, $S\subset T$ and $|S|=|T|$.
    Then $S=T$.
  \item \label{fact:FiniteSubdFinite}
    A subdivision of a finite complex is finite.
  \end{enumerate-(i)}
\end{Lemma}
\begin{proof}
  \ref{fact:SubsetImpliesSimplex} We can assume that \(\tau\) is not a vertex.  Since
  $|\tau|\subset |S|$, $|\overset{\circ}{\tau}|$ intersects \(|\sigma|\) for some
  simplex $\sigma$ in $S \subset T$ and the intersection must be in a common face
  of $\sigma$ and $\tau$.  But the only face of $\tau$ which
  intersects its interior is $\tau$ itself. Thus, $\tau$ is a face of
  $\sigma\in S$. Since $S$ is a complex, $\tau\in S$.
  \ref{fact:SubcomplexEquals} Suppose $\tau\in T$. Then
  $|\tau|\subset |T|=|S|$, so by \ref{fact:SubsetImpliesSimplex},
  $\tau \in S$.
\ref{fact:FiniteSubdFinite} follows by compactness and by Definition
  \ref{def:simplicial_complex}\ref{def:simplicial_complex-3}.
\end{proof}

\begin{Def}\label{def:Star}
  Let $T$ be a simplicial complex and $A$ a set. The \textbf{star} of
  $A$ in $T$, denoted $\St_T(A)$, is the smallest subcomplex $S$ of
  $T$ with $A\cap |T|\subset |S|$. The \textbf{link} of $A$ in $T$,
  denoted $\Lk_T(A)$, is the set
  $$\Lk_{T}(A):=\{\kappa\in \St_T(A)\mid A\cap |\kappa|=\es\}.$$
  It is also a subcomplex of~$T$. For a singleton, we denote
  $\St_{T}(\{z\})=\St_{T}(z)$ and $\Lk_T(\{z\})=\Lk_T(z)$.
\end{Def}

The above definition is written with the following possibility in
mind.  Suppose $S$ is a subcomplex of $T$ and $A\subset |T|$. Then
$\St_{S}(A)$ and $\Lk_S(A)$ are well-defined. In particular, if
$A\cap |S|=\es$, then $\St_{S}(A)=\Lk_S(A)=\es$, and if
$A\subset |S|$, then $\St_S(A)=\St_T(A)$ and $\Lk_S(A)=\Lk_T(A)$.

\begin{Fact}\label{fact:LinkStarVert}
  Suppose $z\in |T|$.  If $\lambda\in\Lk_{T}(z)$, then there is
  $\kappa\in\St_T(z)$ such that $\lambda$ is a face of $\kappa$ and
  $z\in |\kappa|$.  In particular
  $V(\lambda)\cup \{z\}\subset |\kappa|$. Note also that the set of
  vectors $\{\kappa^{-1}(v)-\kappa^{-1}(z)\mid v\in V(\lambda)\}$ is
  linearly independent, because $\kappa^{-1}(z)$ is either in the interior of
  $\bDelta^{\dim(\kappa)}$ or a vertex of \(\kappa\) that does not belong to \(V(\lambda)\), while $\kappa^{-1}(v)$ are vertices of \(\kappa\) for
  $v\in V(\lambda)$.
\end{Fact}

\begin{Def}\label{def:DeltaInComplex}
  Suppose that $T$ is a complex and $x_0,\dots,x_k\in |T|$. If there
  exists $\kappa$ such that $x_0,\dots,x_k\in |\kappa|$, then, abusing notation, let
  $$\Delta[x_0,\dots,x_k]=\kappa\circ\Delta[\kappa^{-1}(x_0),\dots,\kappa^{-1}(x_k)]$$
  (recall Fact and Definition~\ref{def:DeltaInRm}).  Note that
  $\Delta[x_0,\dots,x_k]$ is independent on the choice
  of~$\kappa\in T$ as long as $x_0,\dots,x_k\in |\kappa|$. Like in Definition
  \ref{def:DeltaInRm}, $\Delta[x_0,\dots,x_k]$ is a $k$-simplex if and only if
  $\kappa^{-1}(x_m)-\kappa^{-1}(x_0)$ are linearly independent for $1\le m\le k$,
  a property which is, again, independent on the choice of~$\kappa$.
  In this case, clearly
  \begin{equation}
    V(\Delta[x_0,\dots,x_k])=\{x_0,\dots,x_k\}. \label{eq:VerticesofDelta}
  \end{equation}
  Also note that if $\lambda$ is a simplex in $|T|$ (but not necessarily $\lambda\in T$) with
  $|\lambda|\subset |\kappa|$ for some $\kappa\in T$ such that
  $\kappa^{-1}\lambda$ is rectilinear, then
  $\lambda = \Delta[V(\lambda)]$.
\end{Def}

If $\bar z=(z_i)_{i\in \N}$ is a sequence of elements of a set \(X\),
and $A \subset X$, then denote by
\begin{equation}
  \bar z\cap A \label{eq:CapSeq}
\end{equation}
the subsequence $(z_i)_{i\in I}$ where
$I=\{i\in\N\mid z_i\in A\}$.
If $T$ is a complex and $\bar z\subset |T|$ has the property that
$\bar z\cap C$ is finite for all compact $C\subset |T|$, then we say
that $\bar z$ is \textbf{locally finite}. Note that 
a finite sequence is trivially locally finite.

\begin{Lemma} \label{lem:LocFinEq}
  For a complex $T$ and an infinite sequence $\bar z\subset |T|$ the
  following are equivalent:
  \begin{enumerate-(i)}
  \item \label{lem:LocFinEq-1} $\bar z$ is locally finite,
  \item \label{lem:LocFinEq-2} For all open $U\subset |T|$ such that
    $\overline{U}$ is compact, $\bar z\cap U$ is finite,
  \item \label{lem:LocFinEq-3} For all finite subcomplexes $S\subset T$, $\bar z\cap |S|$ is finite.
  \end{enumerate-(i)}
\end{Lemma}
\begin{proof}
  The equivalence \ref{lem:LocFinEq-1}$\iff$\ref{lem:LocFinEq-2}
  follows because $|T|$ is locally compact and the implication
  \ref{lem:LocFinEq-1}$\Rightarrow$\ref{lem:LocFinEq-3} follows
  because for a finite $S\subset T$, $|S|$ is compact. Let us prove
  \ref{lem:LocFinEq-3}$\Rightarrow$\ref{lem:LocFinEq-1}. Suppose
  $C\subset |T|$ is compact and $\bar z$ satisfies condition
  \ref{lem:LocFinEq-3}.  Let
  $K=\{\kappa\in T\mid |\kappa|\cap C\ne\es \}$.  It is enough to show
  that $K$ is finite. For each $\kappa\in K$, let $U_\kappa$ be an
  open neighbourhood of $\kappa$ which intersects finitely many
  simplices of $T$ which exists by Definition
  \ref{def:simplicial_complex}\ref{def:simplicial_complex-3}. Then
  $\UU=\{U_\kappa\mid \kappa\in K\}$ is an open cover of $C$ and has a
  finite subcover $\UU_*\subset\UU$ by the compactness of $C$. Now
    $$K\subset \{\sigma\in T\mid \sigma\cap U_\kappa\ne \es\text{ for some }U_\kappa\in\UU_*\}$$
    but the right hand side is finite.  
\end{proof}

\begin{Def}\label{def:StellarSubdivision}
  A \textbf{stellar subdivision} of a complex $T$ \textbf{at}
  $z\in |T|$, denoted $T*z$, is the set
  \begin{equation}
    T*z=\big(T\setminus \{\kappa\in T\mid z\in |\kappa|\}\big)\cup \{\Delta[V(\lambda)\cup \{z\}]\mid \lambda\in \Lk_T(z)\}.\label{eq:DefOfStellarSubd}  
  \end{equation}  

\begin{center}
\begin{tikzpicture}[scale=2]

\begin{scope}[shift={(-3,0)}]
    \coordinate (A) at (0,0);
    \coordinate (B) at (1,0);
    \coordinate (C) at (0.5,0.866);
    \coordinate (D) at (1.5,0.866);

    \filldraw[fill=gray!20, draw=black] (A) -- (B) -- (C) -- cycle;
    \filldraw[fill=gray!30, draw=black] (B) -- (D) -- (C) -- cycle;

    \fill (A) circle (1pt) node[below left] {};
    \fill (B) circle (1pt) node[below] {};
    \fill (C) circle (1pt) node[above] {};
    \fill (D) circle (1pt) node[above right] {};

    \node at (0.8,-0.5) {\small Original complex $T$};
\end{scope}

\begin{scope}[shift={(1,0)}]
    \coordinate (A) at (0,0);
    \coordinate (B) at (1,0);
    \coordinate (C) at (0.5,0.866);
    \coordinate (D) at (1.5,0.866);
    \coordinate (O) at (0.5,0.3); 

    \filldraw[fill=green!20, draw=black] (A) -- (B) -- (O) -- cycle;
    \filldraw[fill=green!30, draw=black] (B) -- (C) -- (O) -- cycle;
    \filldraw[fill=green!40, draw=black] (C) -- (A) -- (O) -- cycle;

    \filldraw[fill=gray!30, draw=black] (B) -- (D) -- (C) -- cycle;

    \fill (A) circle (1pt) node[below left] {};
    \fill (B) circle (1pt) node[below] {};
    \fill (C) circle (1pt) node[above] {};
    \fill (D) circle (1pt) node[above right] {};
    \fill (O) circle (1pt) node[right] {$z$};

    \node at (0.8,-0.5) {\small Stellar Subdivision $T * z$};
\end{scope}

\end{tikzpicture}
\end{center}
  
  By Fact~\ref{fact:LinkStarVert}, $T*z$ is well-defined.
  We skip the proof that $T*z$ is a subdivision of~$T$, see \cite{hudson1969pl}
  for a standard reference.
  If $\bar z=(z_0,\dots,z_{n})$ is a finite sequence of elements of
  \(|T|\), define by induction $T*\bar z=T$ for $n=0$ (empty $\bar z$)
  and for $n> 0$ let $T*\bar z=(T*z_0*\cdots*z_{n-1})*z_{n}$. This
  is called \textbf{a $\lo$-stellar subdivision} of~$T$.

  Suppose now that $\bar z=(z_0,z_1,\dots)$ is an infinite locally finite
  sequence. Define
  \begin{equation}
    T*\bar z:=\Cup_{i\in\N}\Cap_{j\ge i}T*\bar z^{j}\label{eq:OmegaSubd}  
  \end{equation}
  where $\bar z^j=(z_0,\dots,z_{j-1})$ is the initial segment of length $j$.
  Equivalently, $\sigma\in T*\bar z$ if and only if there exists $i\in\N$
  such that for all $j\ge i$, $\sigma\in T*\bar z^j$.
  We say that $T*\bar z$ is \textbf{an $\o$-stellar subdivision of~$T$}.
\end{Def}

\begin{center}
\begin{tikzpicture}[scale=2.2, every node/.style={font=\small}]

\begin{scope}[shift={(-2,0)}]
    \coordinate (A) at (0,0);
    \coordinate (B) at (1,0);
    \coordinate (C) at (0.5,0.866);
    \coordinate (D) at (1.5,0.866);

    \filldraw[fill=gray!20] (A)--(B)--(C)--cycle;
    \filldraw[fill=gray!30] (B)--(D)--(C)--cycle;

    \node at (0.8,-0.4) {$T$};
\end{scope}

\begin{scope}[shift={(0,0)}]
    \coordinate (A) at (0,0);
    \coordinate (B) at (1,0);
    \coordinate (C) at (0.5,0.866);
    \coordinate (D) at (1.5,0.866);
    \coordinate (Z0) at (0.5,0.3);

    \filldraw[fill=green!20] (A)--(B)--(Z0)--cycle;
    \filldraw[fill=green!30] (B)--(C)--(Z0)--cycle;
    \filldraw[fill=green!40] (C)--(A)--(Z0)--cycle;

    \filldraw[fill=gray!30] (B)--(D)--(C)--cycle;

    \fill (Z0) circle (1.2pt) node[right] {$z_0$};

    \node at (0.8,-0.4) {$T * \bar z^1$};
\end{scope}

\begin{scope}[shift={(2,0)}]

    \coordinate (A) at (0,0);
    \coordinate (B) at (1,0);
    \coordinate (C) at (0.5,0.866);
    \coordinate (D) at (1.5,0.866);

    \coordinate (Z0) at (0.5,0.3);
    \coordinate (Z1) at (1.1,0.5);
    \coordinate (Z2) at (0.75,0.45);

    \draw[fill=green!20] (A)--(B)--(Z0)--cycle;
    \draw[fill=green!30] (B)--(C)--(Z0)--cycle;
    \draw[fill=green!40] (C)--(A)--(Z0)--cycle;

    \draw[fill=gray!20] (B)--(D)--(Z1)--cycle;
    \draw[fill=gray!30] (D)--(C)--(Z1)--cycle;
    \draw[fill=red!40] (C)--(B)--(Z1)--cycle;

    \draw[fill=red!20] (B)--(Z0)--(Z2)--cycle;
    \draw[fill=red!30] (Z0)--(C)--(Z2)--cycle;
    \draw[fill=red!40] (C)--(B)--(Z2)--cycle;
    \draw (Z1) -- (Z2);

    \fill (Z0) circle (1.2pt) node[below] {$z_0$};
    \fill (Z1) circle (1.2pt) node[right] {$z_1$};
    \fill (Z2) circle (1.2pt) node[above] {$z_2$};

    \node at (0.8,-0.4) {$T * \bar z^3$};

\end{scope}

\end{tikzpicture}
\end{center}

\begin{Lemma}\label{lemma:InfiniteSubdivision}
  If $\bar z$ is locally finite, then $T*\bar z$ is a subdivision
  of~$T$.
\end{Lemma}
\begin{proof}
  Let $\bar z$ be a locally finite sequence. By notational convention as above,
  $\bar z=(z_0,z_1,\dots)$ and $\bar z^j$ is the initial segment of
  $\bar z$ of length $j$. For each $\kappa\in T*\bar z$ select
  $i_\kappa$ such that $\kappa\in T*\bar z^j$ for all $j\ge i_\kappa$. We first
  prove that $T*\bar z$ is a simplicial complex.  Suppose
  $\kappa\in T*\bar z$ and $\lambda$ is a face of $\kappa$.  Let
  $j\ge i_\kappa$.  Since $T*\bar z^j$ is a simplicial complex,
  $\lambda\in T*\bar z^j$.  Since $j$ was chosen arbitrarily,
  $\lambda\in T*\bar z$.  This proves Definition
  \ref{def:simplicial_complex}\ref{def:simplicial_complex-1}.  Suppose now
  $\kappa_0$ and $\kappa_1$ are simplices in $T*\bar z$ with
  $|\kappa_0|\cap |\kappa_1|\ne\es$. Let
  $j\ge \max\{i_{\kappa_0},i_{\kappa_1}\}$. Since $T*\bar z^j$
  satisfies Definition
  \ref{def:simplicial_complex}\ref{def:simplicial_complex-2}, there is
  $\lambda\in T*\bar z^j$ such that
  $|\lambda|=|\kappa_0|\cap|\kappa_1|$ and $\lambda$ is a face of
  both.  Again, by the arbitrariness of $j$, $\lambda\in T*\bar z$.
  This proves Definition
  \ref{def:simplicial_complex}\ref{def:simplicial_complex-2}.  Let us
  prove Definition \ref{def:simplicial_complex}\ref{def:simplicial_complex-3}.
  Suppose $\kappa\in T*\bar z$, and fix \(j \geq i_\kappa\). Then \(T * \bar z^j \subd T\) and so there is \(\kappa' \in T\) such that \(|\kappa| \subset |\kappa'|\). Let $U\subset |T|$ be an open neighbourhood
  of $|\kappa'|$ such that $S=\St_T(U)$ is finite. Such an \(U\) exists by Definition
  \ref{def:simplicial_complex}\ref{def:simplicial_complex-3} applied
  to $T$.  By the local finiteness of $\bar z$ and
  Lemma~\ref{lem:LocFinEq}, there is $i_0>i_\kappa$ such that for all
  $j\ge i_0$, $z_j\notin |S|$ and so it follows by the
  definition of a stellar subdivision and Lemma \ref{lemma:Collection1}\ref{fact:FiniteSubdFinite} that
  $$Z=\{\lambda\in T*\bar z^j\mid |\lambda|\cap U\ne\es\}$$
  is finite. But it also follows that
  $$\{\lambda\in T*\bar z\mid|\lambda|\cap U\ne \es\}=Z.$$
  Noticing that \(U\) is an open neighbourhood of \(|\kappa|\) as well, we are done.

  Next we need to prove that $T*\bar z$ is a subdivision of $T$. We
  will first prove that $|T|=|T*\bar z|$. Suppose $x\in |T|$. Then
  clearly $x\in |T*\bar z^j|$ for all $j$.  By a similar argument as
  above using local finiteness, there is $\kappa\in T*\bar z$ such
  that $x\in |\kappa|$ so we have $|T|\subset |T*\bar z|$. On the
  other hand, suppose $x\in |T*\bar z|$. Then there is $\kappa$ such
  that $\kappa\in T*\bar z^j$ for all $j>i$ for some $i\in\N$ and
  $x\in |\kappa|$. So $x\in |T*\bar z^j|$ for some $j$. Since
  $|T*\bar z^j|=|T|$, $x\in |T|$.

  Finally, we will show that for all $\sigma\in T*\bar z$ there is
  $\kappa\in T$ such that $|\sigma|\subset|\kappa|$ and
  $\kappa^{-1}\sigma$ is rectilinear. Suppose $\sigma\in T*\bar
  z$. Then for some $i$, $\sigma\in T*\bar z^j$ for all $j>i$. Fix
  some $j>i$. Since $T*\bar z^j$ is a subdivision of $T$, there is
  $\kappa\in T$ satisfying the required condition.
\end{proof}

Let $S$ be a simplicial complex in some space $X$ and let
$A\subset X$.  We will denote
\begin{equation}
  S\rest A=\{\sigma\in S\mid |\sigma|\subset A\}. \label{eq:DefOfRest}
\end{equation}

\begin{Lemma}\label{lemma:SubdSubcompl}
  Suppose $S_0$ is a subcomplex of $S$ and $\bar z\subset |S|$ is finite.
  Then $S_0*\bar z$ is a subcomplex of $S*\bar z$
  and $S_0*\bar z=(S*\bar z)\rest |S_0|$.
\end{Lemma}
\begin{proof}
  By applying induction on the length of $\bar z$, we can assume that
  it is a single point $\bar z=z$. If $z\notin |S_0|$, then $S_0*z=S_0$.
  For the first part, we need to show that $S_0*\bar z\subset S*\bar z$,
  so suppose that $\sigma\in S_0*\bar z$.
  By \eqref{eq:DefOfStellarSubd} $\sigma$
  is in one of the following sets:
  \begin{align}
    &S_0\setminus \{\sigma'\in S_0\mid z\in |\sigma'|\}\label{eq:Set1}\\
    &\{\Delta[V(\lambda)\cup \{z\}]\mid \lambda\in \Lk_{S_0}(z)\}\label{eq:Set2}
  \end{align}
  If it is in \eqref{eq:Set1}, then it is in $S_0$ and $z\notin |\sigma|$.
  So it is in $S$ since $S_0\subset S$ but it is not in 
  $$\{\sigma'\in S\mid z\in |\sigma'|\}.$$
  So it is in $S*z$.
  Otherwise it is in \eqref{eq:Set2}. So it is enough to show that
  $\{\Delta[V(\lambda)\cup \{z\}]\mid \lambda\in \Lk_{S_0}(z)\}\subset
  \{\Delta[V(\lambda)\cup \{z\}]\mid \lambda\in \Lk_{S}(z)\}.$
  But this follows if $\Lk_{S_0}(z)\subset\Lk_{S}(z)$, which is trivial to verify
  from Definition~\ref{def:Star}. So $\sigma\in S*z$ and this completes the proof
  of the first claim.

  For the second part, suppose $\sigma\in S_0*z$.  Since $S_0*z$ is a
  subdivision of $S_0$, $|S_0*z|=|S_0|$, so trivially
  $|\sigma|\subset |S_0|$. So it is enough to show that
  $\sigma\in S*\bar z$, but this follows from the first part of this
  lemma.  Suppose $\sigma\in (S*z)\rest |S_0|$. Since $|S_0|=|S_0*z|$,
  $|\sigma|\subset |S_0*z|$. By the first part, $S_0*z$ is a
  subcomplex of $S*z$. Apply Lemma
  \ref{lemma:Collection1}\ref{fact:SubsetImpliesSimplex} with
  $S=S_0*z$, $T=S*z$, and $\tau=\sigma$, so it follows that
  $\sigma\in S_0*z$.
\end{proof}

In the next lemma we collect some useful facts.

\begin{Lemma}\label{lemma:Collection2}
Let $S$ be a complex.
  \begin{enumerate-(i)}
  \item \label{lemma:SwapSetAndSeq} Let
    $A\subset |S|$ and $\bar z\subset |S|$ be finite.  Then
    $$(S*\bar z)\rest A=(S\rest A)*\bar z.$$
  \item\label{fact:ComplexUnchanged} Suppose that $S_0$ is a
    subcomplex of $S$.  Then $z\in |S|\setminus |S_0|$ or
    $z\in V(S_0)$ if and only if $S_0*z=S_0$.
  \item\label{fact:PreserveSimplex} Suppose
    $\sigma\in S$ and $\bar z\subset |S|$ is finite. Then
    $\sigma\in S*\bar z$ if and only if
    $|\bar z|\cap |\sigma|\setminus V(\sigma)=\es$.
  \item \label{fact:PreserveSimplexConv} Conversely, given a finite sequence $\bar z\subset |S|$, if
    $\sigma\in S*\bar z$ and $|\bar z|\cap |\sigma|=\es$, then
    $\sigma\in S$.
  \item \label{fact:OrderChange} Suppose $S_0$ is a subcomplex of $S$
    and $\bar z$ a finite sequence such that $|\bar z|\cap
    |S_0|=\es$. Suppose that $\bar a$ and $\bar b$ are finite
    sequences in $|S|$. Then
    $$S_0*\bar a*\bar b*\bar z=S_0*\bar a*\bar z*\bar b=S_0*\bar z*\bar a*\bar b=S_0*\bar a*\bar b.$$
  \end{enumerate-(i)}
\end{Lemma}
\begin{proof}
  \ref{lemma:SwapSetAndSeq} Let $S_0=S\rest A$. Then it is easy to see
  that $(S*\bar z)\rest A=(S*\bar z)\rest |S_0|$.  Applying
  Lemma~\ref{lemma:SubdSubcompl}, the latter is equal to
  $S_0*\bar z=(S\rest A)*\bar z$.
  
  \ref{fact:ComplexUnchanged} If
  $z\in |S|\setminus |S_0|$, then
  $\{\kappa\in S_0\mid z\in |\kappa|\}=\Lk_{S_0}(z)=\es$.  If
  $z\in V(S_0)$, then
  \begin{equation}
    \label{eq:SetToRemove}
    \{\kappa\in S_0\mid z\in |\kappa|\}
  \end{equation}
  is the
  same set as
  \begin{equation}
    \label{eq:SetToPutBack}
    \{\Delta[V(\lambda)\cup \{z\}]\mid \lambda\in \Lk_{S_0}(z)\}
  \end{equation}
  because $\Delta[V(\kappa)]=\kappa$ and
  $V(\lambda)\cup \{z\}=V(\kappa)$ for all $\lambda\in\Lk_{S_0}(z)$ and for all \(\kappa\) in \eqref{eq:SetToRemove}
  (checking Definitions \ref{def:DeltaInComplex}, \ref{def:Star}). In
  both cases, the direction from left to right follows trivially from
  Definition~\ref{def:StellarSubdivision}.  But if
  $z\in |S_0|\setminus V(S_0)$, then the set \eqref{eq:SetToRemove} is
  non-empty and different from \eqref{eq:SetToPutBack} which proves
  the other direction.  
  
  \ref{fact:PreserveSimplex} By
  \eqref{eq:DefOfRest} it is clear that $\sigma\in S*z$ if and only if
  $\sigma\in (S*z)\rest |\sigma|$.  By
  Lemma~\ref{lemma:Collection2}\ref{lemma:SwapSetAndSeq} the latter is
  equal to $(S\rest|\sigma|)*z=\scl(\sigma)*z$. By
  \ref{fact:ComplexUnchanged}, $\scl(\sigma)*z=\scl(\sigma)$ if and
  only if $\sigma\in\scl(\sigma)$ if and only if $z\in V(\sigma)$ or
  $z\notin |\sigma|$.  
  
  \ref{fact:PreserveSimplexConv} If
  $\sigma\in S*\bar z$ and $\sigma\notin S$, then
  $|\bar z|\cap V(\sigma)\ne\es$ as is clear, again, from
  Definitions~\ref{def:StellarSubdivision}, \ref{def:Star},
  and~\ref{def:DeltaInComplex}. 
  
  \ref{fact:OrderChange} By induction
  on lengths $l(\bar a)$, $l(\bar b)$, and $l(\bar c)$ using
  \ref{fact:ComplexUnchanged}--\ref{fact:PreserveSimplexConv} above. 
  
\end{proof}

\begin{Lemma}\label{lemma:InfSubdLocalFin}
  Suppose $S$ is a complex and $\bar z\subset |S|$ is locally finite.
  Then for every finite subcomplex $S_0\subset S$ there is $i\in\N$
  such that for all $j\ge i$, $S_0*\bar z^j=(S*\bar z)\rest |S_0|$
  (the restriction was defined in~\eqref{eq:DefOfRest}).
\end{Lemma}
\begin{proof}
  By local finiteness there is $i_0$ such that $z_{j_0}\notin |S_0|$
  for all $j_0\ge i_0$. On the other hand for each $\sigma \in S_0$
  there is $i_\sigma$ such that for all $j_{\sigma}\ge i_{\sigma}$,
  $\sigma\in S*\bar z^{j_{\sigma}}$. Since $S_0$ is finite, there is a
  finite upperbound for $\{i_0\}\cup \{i_\sigma\mid \sigma\in
  S_0\}$. Let $i$ be such an upper bound. Fix $j\ge i$.

  Suppose that $\sigma\in S_0*\bar z^j$.  Then by
  Lemma~\ref{lemma:SubdSubcompl}, $\sigma\in S*\bar z^j$.  By
  Lemma~\ref{lemma:Collection2}\ref{fact:PreserveSimplex} and the
  fact that $j\ge i_0$, we have that $\sigma\in S*\bar z^{j'}$ for all
  $j'\ge j$. This implies that $\sigma\in S*\bar z$.  Trivially,
  $|S_0|=|S_0*\bar z^j|$, so $|\sigma|\subset |S_0|$.  This proves the
  inclusion from left to right.

  Suppose $\sigma\in (S*\bar z)\rest |S_0|$. Then
  $|\sigma|\subset |S_0|=|S_0*\bar z^j|$.  On the other hand,
  $\sigma\in S*\bar z$ and since $j\ge i_{\sigma}$, we have
  $\sigma\in S*\bar z^{j}$.  By Lemma~\ref{lemma:SubdSubcompl},
  $S_0*\bar z^j\subset S*\bar z^j$.  Apply
  Lemma~\ref{lemma:Collection1}\ref{fact:SubsetImpliesSimplex} with
  $S=S_0*\bar z^j$, $T=S*\bar z^j$, and $\tau=\sigma$, so we have
  $\sigma\in S_0*\bar z^j$.
\end{proof}

\section{Main result}\label{sec:Main result}

In this section we give a generalization of the main theorem
of \cite{AP24}. 

\begin{Lemma}\label{lemma:ExtedingSubdivisions}
  Let $S$ and $T$ be different finite triangulations of the same space
  $X$. Suppose that $S_0$ and $S_1$ are subcomplexes of $S$, and $T_0$
  and $T_1$ are subcomplexes of $T$ such that $|S_0|=|T_0|$,
  $|S_1|=|T_1|$, $S=S_0\cup S_1$, and $T=T_0\cup T_1$.  Suppose
  further that $\bar s_0\subset |S_0|$, $\bar s_1\subset |S_1|$,
  $\bar t_0\subset |T_0|$, and $\bar t_1\subset |T_1|$ are finite
  sequences such that
  \begin{align}
    S_0*\bar s_0&=T_0*\bar t_0\label{eq:LemmaESAssumption1}\\
    S_1*\bar s_0*\bar s_1&=T_1*\bar t_0*\bar t_1.\label{eq:LemmaESAssumption2}
  \end{align}
  Then $S*\bar s_0*\bar s_1=T*\bar t_0*\bar t_1$.
\end{Lemma}
\begin{proof}
  We will prove the statement first for $\bar s_0=\bar t_0=\es$.  The
  assumptions of the lemma now can be re-written as follows:
  \begin{align}
    S_0&=T_0\label{eq:LemmaESAssumption11}\\
    S_1*\bar s_1&=T_1*\bar t_1,\label{eq:LemmaESAssumption22}
  \end{align}
  and the conclusion as $S*\bar s_1=T*\bar t_1$.  Suppose
  $\sigma\in S*\bar s_1$. We will show that then
  $\sigma\in T*\bar t_1$. This, by symmetry, will imply the claim.
  By Lemma~\ref{lemma:SubdSubcompl},
  $S*\bar s_1=(S_0*\bar s_1)\cup (S_1*\bar s_1)$.  If
  $\sigma\in S_1*\bar s_1=T_1*\bar t_1\subset T*\bar t_1$, then we are
  done.  So assume that $\sigma\notin S_1*\bar s_1$ and so
  $\sigma\in S_0*\bar s_1$. Since
  $S_0*\bar s_1$ is a subdivision of $S_0$, there is $\sigma'\in S_0$
  such that $|\sigma|\subset |\sigma'|$ and
  $\sigma\in \sigma'*\bar s_1$ by Lemma \ref{lemma:SubdSubcompl}. But
  $\sigma'\in T_0\subset T$, so $\sigma'*\bar s_1$ is also a
  subcomplex of $T*\bar t_1$ (also by Lemma \ref{lemma:SubdSubcompl}), so
  $\sigma\in T*\bar t_1$.

  Now suppose $\bar s_0\ne\es$ or $t_0\ne\es$. Define $S'=S*\bar s_0$,
  $T'=T*\bar t_0$, $S'_0=S_0*\bar s_0$, $T'_0=T_0*\bar t_0$, $S_1'=S_1*\bar s_0$,
  and $T_1'=T_1*\bar t_0$. By Lemma~\ref{lemma:SubdSubcompl}
  we have $S'=S'_0\cup S'_1$ and
  $T'=T'_0\cup T'_1$, so we can apply the first part
  of the proof with $S_0=S_0'$, $T_0=T_0'$,
  $S_1=S'_1$, and $T_1=T_1'$.
\end{proof}

\begin{thm}[Common stellar subdivision for infinite complexes]\label{thm:CommonStellarSubdivision}
  Suppose that two (possibly infinite) triangulations $S$ and $T$ of the same space $X$
  have a common subdivision. Then there are locally finite sequences
  $\bar s,\bar t\subset X$ such that $S*\bar s=T*\bar t$.
\end{thm}
\begin{proof}
  This was proved in \cite{AP24} for finite
  complexes $S$ and $T$ in which case the sequences
  $\bar s$ and $\bar t$ are finite too. So we assume that
  $S$ and $T$ are infinite. 
  We will begin by treating the case $S\subd T$:
  \begin{claim}\label{claim:CommonStellar}
    Suppose $S\subd T$. Then there are $\bar s$ and $\bar t$
    such that $S*\bar s=T*\bar t$.
  \end{claim}
  \begin{proof}
    Let $T_0, T_1,\cdots$ be a locally finite sequence of finite
    subcomplexes of $T$ such that $T=\Cup_{i\in\N}T_i$ and such that
    $T_0$ is a single vertex in $T$ (so also in $S$).  For example,
    obtain this sequence by setting $T_0=\{v_0\}$, with
    $v_0 \in V(T)$, and $T_i=\St_T(v_i)$ for $i\ge 1$, where
    $(v_i)_{i\in\N}$ is an enumeration of $V(T)$. Let
    $\hat T_i=\Cup_{j\le i}T_j$.  Let $S_i:=S\rest T_i$ and
    $\hat S_i=S\rest \hat T_i$.

    By induction on $i$ we will find sequences $\bar s^i$ and
    $\bar t^i$ such that $\hat S_i*\bar s^i=\hat T_i*\bar t^i$.
    Let $\bar s^0=\bar t^0$ be the empty
    sequence. So $\hat S_0=\{v_0\}= \hat T_0$ and
    $\hat S_0*\bar s^0= \hat T_0*\bar t^0$.  Note that since $S$ is a
    subdivision of $T$, it follows that $|S_i|=|T_i|$ and
    $|\hat S_i|=|\hat T_i|$.  Suppose now that $\bar s^i$ and
    $\bar t^i$ have been defined such that
    $\bar s^i,\bar t^i\subset |\hat T_i|$ and
    $\hat S_i*\bar s^i=\hat T_i*\bar t^i$.
    By
    \cite{AP24} there are $\bar s$ and $\bar t$ such that
    $S_{i+1}*\bar s^i*\bar s=T_{i+1}*\bar t^i*\bar t$.
    We can
    now apply Lemma~\ref{lemma:ExtedingSubdivisions} with
    $S_0=\hat S_{i}$, $S_1=S_{i+1}$, $S=\hat S_{i+1}$, $T_0=\hat T_i$,
    $T_1=T_{i+1}$, $T=\hat T_{i+1}$, $\bar s_0=\bar s^{i}$,
    $\bar s_1=\bar s$, $\bar t_0=\bar t^{i}$, and $\bar t_1=\bar t$.
    Then we have
    $$\hat S_{i+1}*\bar s^i*\bar s = \hat S_{i+1}*\bar t^i*\bar t.$$
    Let $\bar s^{i+1}:=\bar s^i\cat \bar s$ and
    $\bar t^{i+1}\cat \bar t$.  At the limit we set
    $\bar s=\Cup_{i\in\N}\bar s^i$ and $\bar t=\Cup_{i\in\N}\bar t^i$.
    Let us show that now $S*\bar s=T*\bar t$.  Suppose
    $\sigma\in S*\bar s$. Then by \eqref{eq:OmegaSubd}, for some
    $i\in\N$, $\sigma\in S*\bar s^j$ for all $j>i$.  Let $i'>i$ be so
    large that $|\sigma|\subset \hat S_{i'}$ and for all $j>i'$,
    $\bar s^j\cap \hat S_{i'}=\bar s^{i'}\cap \hat S_{i'}$.  This is
    possible by the construction. Now $\sigma\in \hat S_{i'}*\bar s^j$
    for all $j>i'$ and so by Lemma~\ref{lemma:SubdSubcompl}, also
    $\sigma\in \hat S_j*\bar s^j=\hat T_j*\bar t^j\subset T*\bar t^j$
    for all $j>i'$. Therefore $\sigma\in T*\bar t$. The reverse inclusion follows
    symmetrically.
  \end{proof}
  Now suppose $S$ and $T$ are any triangulations which have a common
  subdivision $Z$.
  By the claim applied to \(Z \subd T\), let
  $\bar z'$ and $\bar t$ be such that $Z*\bar z'=T*\bar t$. Then
  \(Z *\bar z' \subd Z\), so by transitivity, also
  \(Z *\bar z' \subd S\). Now apply the claim to the latter pair to
  get $\bar z$ and $\bar s$ such that $(Z*\bar z')*\bar z=S*\bar
  s$. We now have
  \begin{equation}
    (T*\bar t)*\bar z=S*\bar s.\label{eq:omegapomegaeq} 
  \end{equation}
  There is only one problem. The sequence $\bar t$ can be infinite, so
  the concatenation $\bar t\cat \bar z$ would be a sequence of length
  $\o+\o$. So instead, we have to ``concatenate'' them in a different
  way.  Let $T_0\subset T_1\subset\cdots$ be an increasing sequence of
  finite subcomplexes of $T$ such that $T=\Cup_{i\in\N}T_i$,
  $z_0\in T_0$, and
  $z_{i+1}\in |T_{i}|\setminus \St_{T}(|T|\setminus |T_{i}|)$ for
  all $i\in\N$. The last condition is essentially saying that $z_{i+1}$
  is in $|T_{i}|$ but not ``close to the boundary of $|T_{i}|$''. For
  each $i\in\N$, let $k(i)\in \N$ be such that for all
  $j\ge k(i)$, $z_j\notin |T_i|$, $t_j\notin |T_i|$, and (in particular)
  \begin{equation}
    \label{eq:DefOfki}
    T_i*\bar t^j=(T*\bar t) \rest T_i.
  \end{equation}
  Such $k(i)$ exists for all $i$ by local finiteness of $\bar z$ and
  $\bar t$ and Lemma~\ref{lemma:InfSubdLocalFin}.  W.l.o.g. assume
  that for all $i$, $k(i)>i$ and for all $i<i'$, $k(i)<k(i')$.  Then
  for all $i<j$ define $\bar t^{i,j}:=(t_{i},\ldots, t_{j-1})$ and let
  $$\bar w:=\bar t^{0,k(0)}\cat (z_0)\cat \bar t^{k(0),k(1)}\cat (z_1)\cat \bar t^{k(1),k(2)}\cat (z_2)\cat \cdots.$$
  By \eqref{eq:omegapomegaeq}, it is now enough to show that
  $(T*\bar t)*\bar z=T*\bar w$. First we need a claim.
  \begin{claim}\label{claim:wtz_order}
    For all $i$, $T_i*\bar t^{k(i)}*\bar z^{i}=T_i*\bar w^{k(i)+i}$.
  \end{claim}
  \begin{proof}
    We will prove this by induction on $i$. If $i=0$, then
    $\bar z^{0}=\es$ and $\bar t^{k(0)}=\bar w^{k(0)}=\bar w^{k(0)+0}$,
    so the claim is trivial. Suppose we have proved it for $i$
    and we will now prove it for $i+1$:
    $$T_{i+1}*\bar t^{k(i+1)}*\bar z^{i+1}=T_{i+1}*\bar w^{k(i+1)+i+1}.$$
    Note that from the induction hpothesis, it follows from Lemma~\ref{lemma:SubdSubcompl},
    that for any subcomplex $S\subset T_{i+1}$ we have
    \begin{equation}
      S*\bar t^{k(i)}*\bar z^{i}=S*\bar w^{k(i)+i}.\label{eq:AnySubdivision}    
    \end{equation}
    Suppose $\sigma\in T_{i+1}*\bar t^{k(i+1)}*\bar z^{i+1}$.  Let
    $\sigma'\in T_{i+1}$ be such that $|\sigma|\subset |\sigma'|$.  In
    particular $\sigma\in \scl(\sigma')*\bar t^{k(i+1)}*\bar z^{i+1}$ by
    Lemma~\ref{lemma:SubdSubcompl} (here $\scl(\sigma')$ is the
    simplicial closure of $\sigma'$ from
    Definition~\ref{def:simplicial_complex}, i.e. $\sigma'$ together
    with all its faces).  If $\sigma'\in T_i$, then
    $|\bar t^{k(i),k(i+1)}|\cap |\sigma'|=\es$, so by
    Lemma~\ref{lemma:Collection2}\ref{fact:OrderChange} and by
    \eqref{eq:AnySubdivision} we have
    \begin{align*}
      \scl(\sigma')*\bar t^{k(i+1)}*\bar z^{i+1}&=\scl(\sigma')*\bar t^{k(i)}*\bar t^{k(i),k(i+1)}*\bar z^{i}*z_i\\
    &=\scl(\sigma')*\bar t^{k(i)}*\bar z^{i}*z_i*\bar t^{k(i),k(i+1)}\\
    &=\scl(\sigma')*\bar w^{k(i)+i}*z_i*t^{k(i),k(i+1)}\\
    &=\scl(\sigma')*\bar w^{k(i+1)+i+1}\\
    &\subset T_{i+1}*\bar w^{k(i+1)+i+1}
    \end{align*}
    implying $\sigma\in T_{i+1}*\bar w^{k(i+1)+i+1}$.  Otherwise,
    $\sigma'\notin T_i$ and so
    $\sigma'\subset \St_T(|T_{i+1}|\setminus |T_i|)$.  So we can safely
    assume that $|\bar z^{i+1}|\cap |\sigma'|=\es$. So applying
    Lemma~\ref{lemma:Collection2}\ref{fact:OrderChange} in a slightly
    different way, we get the same derivation as above.  Now suppose
    conversely that
    $\sigma\in T_{i+1}*\bar w^{k(i+1)+i+1}=T_{i+1}*\bar
    w^{k(i)+i}*z_i*\bar t^{k(i),k(i+1)}$.  But then the same argument
    as above (using the induction hypothesis and Lemma
    \ref{lemma:Collection2}\ref{fact:OrderChange}) in a reverse way
    yields $\sigma\in T_{i+1}*\bar t^{k(i+1)}*\bar z^{i+1}$ so we are
    done.
  \end{proof}
  
  \begin{claim}\label{claim:wtz_order2}
    Suppose $i_1\ge i_0$. Then $T_{i_0}*\bar t^{k(i_1)}*\bar z^{i_1}=T_{i_0}*\bar w^{k(i_1)+i_1}$.
  \end{claim}
  \begin{proof}
    Let $\sigma\in T_{i_0}*\bar t^{k(i_1)}*\bar z^{i_1}$.
    Since $T_{i_0}$ is a subcomplex of $T_{i_1}$, by Lemma~\ref{lemma:InfSubdLocalFin}
    and Claim \ref{claim:wtz_order},
    $$T_{i_0}*\bar t^{k(i_1)}*\bar z^{i_1}\subset T_{i_1}*\bar t^{k(i_1)}*\bar z^{i_1}=T_{i_1}*\bar w^{k(i_1)+i_1},$$
    so $\sigma\in T_{i_1}*\bar w^{k(i_1)+i_1}$. Again applying
    Lemma~\ref{lemma:InfSubdLocalFin} we know that
    $T_{i_0}*\bar w^{k(i_1)+i_1}$ is a subcomplex of
    $T_{i_1}*\bar w^{k(i_1)+i_1}$ and we also know that
    $|\sigma|\subset |T_{i_0}|=|T_{i_0}*\bar w^{k(i_1)+i_1}|$. Apply
    Lemma~\ref{lemma:Collection1}\ref{fact:SubsetImpliesSimplex} with
    $S=T_{i_0}*\bar w^{k(i_1)+i_1}$, $T=T_{i_1}*\bar w^{k(i_1)+i_1}$, and
    $\tau=\sigma$. Thus, $\sigma\in T_{i_0}*\bar w^{k(i_1)+i_1}$.  This
    proves inclusion from left to right. The reverse inclusion follows
    by a symmetric argument.
  \end{proof}
  Now suppose $\sigma\in (T*\bar t)*\bar z$. We want to show that
  $\sigma\in T*\bar w$. We need to find $i\in\N$ such that for all
  $j>i$, $\sigma\in T*\bar w^{j}$. Let $i_0$ be so large that
  $|\sigma|\subset |T_{i_0}|$ and for all $j\ge i_0$,
  $\sigma\in (T*\bar t)*\bar z^j$.
  Let $i_1=k(i_0)$ and $i=k(i_1)+i_1$. Now
  \begin{equation}
    \label{eq:WhatIswi}
    \bar w^{i}=\bar t^{0,k(0)}\cat (z_0)\cat \bar t^{k(0),k(1)}\cat (z_1)\cat \bar t^{k(1),k(2)}\cat (z_2)\cat \cdots\cat (z_{i_1-1})\cat \bar t^{k(i_1-1),k(i_1)}.
  \end{equation}
  Suppose $j\ge i$. Then by \eqref{eq:WhatIswi}, every
  $w\in |\bar w^{j}|\setminus |\bar w^{i}|$ either equals to $t_{l}$
  for $l\ge i>k(i_1)>k(i_0)$ or to $z_l$ for $l\ge i_1=k(i_0)$, so $w$
  is outside of $T_{i_0}$ and so by
  Lemma~\ref{lemma:Collection2}\ref{fact:PreserveSimplex}, if
  $\sigma\in T_{i_0}*\bar w^{i}$, then
  $\sigma\in T_{i_0}*\bar w^{j}\subset T*\bar w^{j}$ where the last
  inclusion follows from Lemma~\ref{lemma:SubdSubcompl}.  Thus, it is
  enough to show that
  $\sigma\in T_{i_0}*\bar w^{i}=T_{i_0}*\bar w^{k(i_1)+i_1}$.  Since
  $i_1=k(i_0)>i_0$, we have by the definition of $i_0$ that
  $\sigma\in (T*\bar t)*\bar z^{i_1}$.  But also
  $|\sigma|\subset |T_{i_0}|$, so
  \begin{align*}
    \sigma\in \left((T*\bar t)*\bar z^{i_1}\right)\rest |T_{i_0}|&=\left((T*\bar t)\rest |T_{i_0}|\right)*\bar z^{i_1}&&\text{ by Lemma~\ref{lemma:Collection2}\ref{lemma:SwapSetAndSeq}}\\
    &=\left(T_{i_0}*\bar t^{k(i_1)}\right)*\bar z^{i_1}&&\text{ by \eqref{eq:DefOfki}}\\
    &=T_{i_0}*\bar t^{k(i_1)}*\bar z^{i_1}\\
    &=T_{i_0}*w^{k(i_1)+i_1}&&\text{ by Claim~\ref{claim:wtz_order2}}.
  \end{align*}
  In the second equality above, we can indeed apply \eqref{eq:DefOfki}
  because $k(i_1)>i_1=k(i_0)$. This completes the proof of
  $(T*\bar t)*\bar z\subset T*\bar w$.  By
  Lemma~\ref{lemma:Collection1}\ref{fact:SubcomplexEquals},
  $(T*\bar t)*\bar z=T*\bar w$.
\end{proof}

\bibliographystyle{alpha}
\bibliography{bibliography}

\end{document}